\documentclass[10pt,reqno]{amsart}

\usepackage{a4wide}
\usepackage{dsfont} 
\usepackage{amssymb,amsmath}
\usepackage{mathrsfs}

\newtheorem{proposition}{Proposition}[section]
  \newtheorem{theorem}[proposition]{Theorem}
  \newtheorem{corollary}[proposition]{Corollary}
  \newtheorem{lemma}[proposition]{Lemma}
\theoremstyle{remark}

\newcommand{\cst}{\ifmmode\mathrm{C}^*\else{$\mathrm{C}^*$}\fi}
\newcommand{\st}{\;\vline\;}
\newcommand{\CC}{\mathbb{C}}
\newcommand{\GG}{\mathbb{G}}
\newcommand{\tens}{\otimes}
\newcommand{\id}{\mathrm{id}}
\newcommand{\comp}{\circ}
\newcommand{\ph}{\varphi}
\newcommand{\rh}{\varrho}
\newcommand{\I}{\mathds{1}}
\newcommand{\HH}{\mathbb{H}}
\newcommand{\cG}{\mathscr{G}}
\newcommand{\cF}{\mathscr{F}}
\newcommand{\sC}{\mathsf{C}}
\newcommand{\sA}{\mathsf{A}}
\newcommand{\cA}{\mathscr{A}}
\newcommand{\cH}{\mathscr{H}}
\newcommand{\h}{\boldsymbol{h}}
\newcommand{\is}[2]{\left\langle#1\,\vline\,#2\right\rangle}
\newcommand{\hh}[1]{\widehat{#1}}

\DeclareMathOperator{\C}{C}
\DeclareMathOperator{\cL}{\mathscr{L}}
\DeclareMathOperator{\Pol}{Pol}

\numberwithin{equation}{section}

\author{Pawe{\l} Kasprzak}
\address{Department of Mathematical Methods in Physics, Faculty of Physics, University of Warsaw, Poland}
\email{pawel.kasprzak@fuw.edu.pl}

\author{Piotr M.~So{\l}tan} 
\address{Department of Mathematical Methods in Physics, Faculty of Physics, University of Warsaw, Poland}
\email{piotr.soltan@fuw.edu.pl}

\author{Stanis\l{}aw L.~Woronowicz}
\address{Institute of Mathematics, University of Bialystok, Poland and Department of Mathematical Methods in Physics, Faculty of Physics, University of Warsaw, Poland}
\email{stanislaw.woronowicz@fuw.edu.pl}

\thanks{Supported by National Science Centre (NCN) grant no.~2011/01/B/ST1/05011}

\title{Quantum automorphism groups of finite quantum groups are classical}

\begin{document}

\begin{abstract}
In a recent paper of Bhowmick, Skalski and So\l{}tan the notion of a quantum group of automorphisms of a finite quantum group was introduced and, for a given finite quantum group $\GG$, existence of the universal quantum group acting on $\GG$ by automorphisms was proved. We show that this universal quantum group is in fact a classical group. The key ingredient of the proof is the use of multiplicative unitary operators, and we include a thorough discussion of this notion in the context of finite quantum groups.
\end{abstract}

\maketitle

\section{Finite quantum groups and associated multiplicative unitaries}

In the literature on non-commutative geometry and quantum groups the definition of a quantum group is still not established. Some authors prefer the point of view that a quantum group is nothing else than a Hopf algebra or an appropriate topological analog of a Hopf algebra. Others make a point of explicitly invoking the duality between spaces and commutative algebras and say that a quantum group is an object $\GG$ of the category dual to that of a certain class of algebras (say \cst-algebras) with additional properties. These properties are only expressible in therms of the object $\GG$ viewed back in the category of algebras and written as e.g.~$\C(\GG)$. Then the fact that $\GG$ is a ``quantum group'' amounts to a Hopf algebra-like structure on $\C(\GG)$. 

In this paper we will try to stay on middle ground between these approaches. A \emph{finite quantum group} will simply be a finite-dimensional Hopf $*$-algebra $(\cA,\Delta)$ over $\CC$ such that $\cA$ admits a faithful positive functional. In this situation $\cA$ is in fact a \cst-algebra, so in particular it is semisimple. On the other hand the compact quantum group of automorphisms of $(\cA,\Delta)$ will be denoted by $\HH$ and the notation will be that of e.g.~\cite{BS,bhss}).

The \emph{Haar functional} or the \emph{Haar measure} of $(\cA,\Delta)$ is a normalized positive linear functional $\h$ on $\cA$  such that
\[
\begin{aligned}
(\h\tens\id)\Delta(a)&=(\id\tens\h)\Delta(a)=\h(a)\I,&\quad{a}\in\cA.
\end{aligned}
\]
Existence of $\h$ is a classical result in Hopf algebras (\cite{LarSweed,Sweed}). The functional $\h$ is faithful (\cite{vdfin}) and consequently introduces a scalar product on $\cA$ via
\begin{equation}\label{scal}
\begin{aligned}
\is{a}{b}&=\h(a^*b),&\quad{a,b}\in\cA.
\end{aligned}
\end{equation}
It is clear from formula \eqref{scal} that we adopt the convention that scalar products are antilinear in the first variable and linear in the second. Let $\cH$ denote the finite-dimensional Hilbert space obtained by considering $\cA$ with scalar product \eqref{scal}. The algebra $\cA$ acts on $\cH$ by left multiplication and we denote by $\sA$ the image of $\cA$ in $\cL(\cH)$ in this representation. 

In what follows the algebras $\cA$ and $\sA$ will be identified and the symbols $\cA$ and $\sA$ will be used interchangeably depending on the context.

Consider the linear map
\begin{equation}\label{linear}
\cA\tens\cA\ni(a\tens{b})\longmapsto\Delta(a)(\I\tens{b})\in\cA\tens\cA
\end{equation}
It is well known that this map is invertible with inverse given by $a\tens{b}\longmapsto\bigl((\id\tens{S})\Delta(a)\bigr)(\I\tens{b})$, where $S$ is the antipode of $(\cA,\Delta)$. Let $W$ be the same map only considered as a mapping $\cH\tens\cH$. The difference is almost negligible, but it makes sense to keep the distinction. The invariance of $\h$ immediately implies that $W\in\cL(\cH\tens\cH)$ is unitary:
\[
\begin{split}
\is{W(a_1\tens{b_1})}{W(a_2\tens{b_2})}&=(\h\tens\h)\bigl((\I\tens{b_1^*})\Delta(a_1)^*\Delta(a_2)(\I\tens{b_2})\bigr)\\
&=(\h\tens\h)\bigl((\I\tens{b_1^*})\Delta(a_1^*a_2)(\I\tens{b_2})\bigr)\\
&=\h\bigl(b_1^*\bigl((\h\tens\id)\Delta(a_1^*a_2)\bigr)b_2\bigr)\\
&=\h(b_1^*b_2)\h(a_1^*a_2)=\is{a_1\tens{b_1}}{a_2\tens{b_2}}
\end{split}
\]
for all $a_1,a_2,b_1,b_2\in\cH$.

Given $a,b\in\cH$ we define $\omega_{a,b}$ to be the linear functional on $\cL(\cH)$ given by
\[
T\longmapsto\is{a}{Tb}.
\]
Such functionals span all of $\cL(\cH)^*$. Using the fact that $\cL(\cH\tens\cH)\cong\cL(\cH)\tens\cL(\cH)$ in a canonical way, we can consider the \emph{left slice} $(\omega_{a,b}\tens\id)W$ of $W$ with $\omega_{a,b}$. It is easily checked that $(\omega_{a,b}\tens\id)W$ is the operator on $\cH$ which acts as left multiplication by 
\[
(\h\tens\id)\bigl((a^*\tens\I)\Delta(b)\bigr)\in\cA.
\]
Noting that just like \eqref{linear}, the map $a\tens{b}\mapsto(a\tens\I)\Delta(a)$ is surjective (Hopf $*$-algebras are regular, \cite{mha}) we find that 
\[
\bigl\{(\omega\tens\id)W\st\omega\in\cL(\cH)^*\bigr\}=\sA.
\]

Another important feature of $W$ is that it is a \emph{multiplicative unitary}, i.e.~the \emph{pentagon equation} holds: we have
\begin{equation}\label{penta}
W_{23}W_{12}W_{23}^*=W_{12}W_{13}
\end{equation}
holds. The notation of \eqref{penta} requires explanation: we define $W_{12}=W\tens\I_\cH$, $W_{23}=\I_\cH\tens{W}$ and $W_{13}=(\Sigma\tens\I_\cH)W_{23}(\Sigma\tens\I_\cH)$, where $\Sigma$ denotes the flip map $\cH\tens\cH\ni(a\tens{b})\mapsto(b\tens{a})\in\cH\tens\cH$. This kind of notation is known as the \emph{leg numbering notation}. In what follows we will use this notation with up to five tensor factors and not only for linear operators on tensor products, but also for elements of multiplie tensor products of algebras. Multiplicative unitaries were introduced in \cite{BaajSkand} and the theory was later developed in many papers (e.g.~\cite{finmu,mu,remmu,mmu}).

The comultiplication $\Delta$ is \emph{implemented} by $W$ in the following sense: given $a\in\sA\cong\cA$ the operator on $\cH\tens\cH$ of left multiplication by $\Delta(a)$ is 
\[
W(a\tens\I)W^*.
\]
Thus transporting $\Delta$ from $\cA$ onto $\sA$ we obtain a comultiplication on $\sA$ which by \eqref{penta} satisfies
\[
(\id\tens\Delta)W=W_{12}W_{13}.
\]

Similarly, transporting the antipode $S$ onto $\sA$ we find that
\begin{equation}\label{SW}
(\id\tens{S})W=W^*.
\end{equation}
Indeed, given $a,a',b,b'\in\cH$ we can calculate
\[
\begin{split}
\is{a'}{S\bigl((\omega_{a,b}\tens\id)W\bigr)b'}&=\h\bigl({a'}^*S((\h\tens\id)((a^*\tens\I)\Delta(b)))b'\bigr)\\
&=(\h\tens\h)\bigl((a^*\tens{a'}^*)((\id\tens{S})\Delta(b))(\I\tens{b'})\bigr)\\
&\is{a'}{(\omega_{a,b}\tens\id)(W^*)b'}.
\end{split}
\]

Let
\[
\hh{\sA}=\bigl\{(\id\tens\omega)W\st\omega\in\cL(\cH)^*\bigr\}.
\]
Then $\hh{\sA}$ is a vector subspace of $\cL(\cH)$ and $W\in\hh{\sA}\tens\sA$. In particular, as $\sA\cong\cA$, the space $\hh{\sA}$ is equal to $\bigl\{(\id\tens\omega)W\st\omega\in\cA^*\bigr\}$. Moreover
\[
\cG\colon\cA^*\ni\ph\longmapsto(\id\tens\ph)W\in\hh{\sA}
\]
is an isomorphism of vector spaces. Indeed, it is clearly surjective and $(\id\tens\ph)W=0$ implies $\ph\bigl((\id\tens\omega)W\bigr)=0$ for all $\omega\in\cL(\cH)^*$, so $\ph=0$.

As $\cA$ is a Hopf $*$-algebra, so is $\cA^*$ (this is the \emph{dual Hopf $*$-algebra} of $(\cA,\Delta)$, \cite{dual}, cf.~also \cite[Section 4]{afgd}). It turns out that the linear map $\cG$ is also an algebra isomorphism: take $\ph,\psi\in\cA$. Then
\[
\cG(\ph\cdot\psi)=(\id\tens\ph\cdot\psi)W=(\id\tens\ph\tens\psi)(\id\tens\Delta)W=(\id\tens\ph\tens\psi)W_{12}W_{13}=\cG(\ph)\cG(\psi).
\]
Furthermore $\cG$ is a $*$-map. Recall that the involution on $\cA^*$ is defined by $\ph^*(a)=\overline{\ph\bigl(S(a)^*\bigr)}$. Thus, using \eqref{SW} and the fact that for any $X\in\cL(\cH\tens\cH)$
\[
\bigl((\omega_{a,b}\tens\id)X\bigr)^*=(\omega_{b,a}\tens\id)(X^*),
\]
we compute for $a,b\in\cH$
\[
\begin{split}
\is{a}{\bigl((\id\tens\ph^*)W\bigr)b}&=\ph^*\bigl((\omega_{a,b}\tens\id)W\bigr)\\
&=\overline{\ph\bigl(S((\omega_{a,b}\tens\id)W)^*\bigr)}\\
&=\overline{\ph\bigl(((\omega_{a,b}\tens\id)(W^*))^*\bigr)}\\
&=\overline{\ph\bigl((\omega_{b,a}\tens\id)W\bigr)}\\
&=\overline{\is{b}{\bigl((\id\tens\ph)W\bigr)a}}\\
&=\is{\bigl((\id\tens\ph)W\bigr)a}{b}\\
&=\is{a}{\bigl((\id\tens\ph)W\bigr)^*b},
\end{split}
\]
so $\cG(\ph^*)=\cG(\ph)^*$ (we used the fact that $S$ is a $*$-map).

The coproduct of $\cA^*$ is defined by $\ph\mapsto\ph\comp\mu$, where $\mu\colon\cA\tens\cA\to\cA$ is the multiplication map (we are using the canonical isomorphism $(\cA\tens\cA)^*\cong\cA^*\tens\cA^*$). This map can also be transported to $\hh{\sA}$ using $\cG$: we define $\hh{\Delta}\colon\hh{\sA}\to\hh{\sA}\tens\hh{\sA}$ by
\begin{equation}\label{Delhat1}
\hh{\Delta}\bigl((\id\tens\ph)W\bigr)=(\cG\tens\cG)(\ph\comp\mu).
\end{equation}
To compute the right hand side of \eqref{Delhat1} note that for any $\psi_1,\psi_2\in\cA^*$ we have
\[
(\cG\tens\cG)(\psi_1\tens\psi_2)=(\id\tens\psi_1\tens\id\tens\psi_2)(W\tens{W}).
\]
By linearity this means that for $\Phi\in\cA^*\tens\cA^*$
\[
(\cG\tens\cG)(\Psi)=(\id\tens\id\tens\Phi)(\id\tens\sigma\tens\id)(W\tens{W}),
\]
where $\sigma$ is the flip $\sA\tens\hh{\sA}\to\hh{\sA}\tens\sA$. Thus, using \eqref{penta} we find that
\[
\begin{split}
(\cG\tens\cG)(\ph\comp\mu)&=(\id\tens\id\tens\ph)(W_{13}W_{23})\\
&=(\id\tens\id\tens\ph)(W_{12}^*W_{23}W_{12})\\
&=W^*\bigl(\I\tens[(\id\tens\ph)W]\bigr)W.
\end{split}
\]
It follows that
\begin{equation}\label{Delhat2}
\hh{\Delta}(x)=W^*(\I\tens{x})W
\end{equation}
for all $x\in\hh{\sA}$ (in particular the right hand side of \eqref{Delhat2} belongs to $\hh{\sA}\tens\hh{\sA}$). By the pentagon equation \eqref{penta} we have
\[
(\hh{\Delta}\tens\id)W=W_{13}W_{23}.
\]

\section{Actions of compact quantum groups on finite-dimensional algebras}

Let $\cA$ be a finite-dimensional \cst-algebra and let $\HH$ be a compact quantum group. An \emph{action} of $\HH$ on $\cA$ is a unital $*$-homomorphism $\alpha\colon\cA\to\cA\tens\C(\HH)$ such that
\[
(\alpha\tens\id)\comp\alpha=(\id\tens\Delta_\HH)\comp\alpha,
\]
where $\Delta_\HH$ is to comultiplication on $\C(\HH)$ and
\[
\alpha(\cA)\bigl(\I\tens\C(\HH)\bigr)=\cA\tens\C(\HH).
\]
The theory of action of compact quantum groups on \cst-algebras is by now very well established (see e.g.~\cite{Sym,boca} or \cite[Section 6]{PodMu}). A very important part of the theory is concerned with the purely algebraic aspects of such an action. In \cite{Sym} existence of a certain dense $*$-subalgebra of a \cst-algebra carrying an action of a compact quantum group was shown. This subalgebra is sometimes referred to as the \emph{algebraic core} or \emph{Podle\'s core} of the action (\cite{dz}). In our case, with $\cA$ finite-dimensional, the Podle\'s core of the action $\alpha$ is equal to $\cA$ and consequently the image of $\alpha$ is contained in the (algebraic) tensor product of $\cA$ with $\Pol(\HH)$, the unique dense Hopf $*$-subalgebra of $\C(\HH)$ (\cite[Theorem 2.2]{cqg}, \cite[Theorem 5.1]{bmt}). In particular we can apply $\id\tens{S_\HH}$ to $\alpha(a)$ for any $a\in\cA$.

Let $\rho$ be a state on $\cA$. We say that $\rho$ is \emph{invariant} for the action $\alpha$ if
\[
\begin{aligned}
(\rho\tens\id)\alpha(a)&=\rho(a)\I,&\quad{a}\in\cA.
\end{aligned}
\]

In the next section we will need the following result which by analogy with \cite[Remark 1.8]{mnw} may be called \emph{strong right invariance} of an invariant state (cf.~also \cite[Formula (1.3)]{bhss}).

\begin{lemma}\label{strinv}
Let $\alpha$ be an action of a compact quantum group $\HH$ on a finite-dimensional \cst-algebra $\cA$. Let $\h$ be a state on $\cA$ invariant for $\alpha$. Further let $S_\HH$ be the antipode of $\HH$. Then
\[
(\h\tens\id)\bigl(\alpha(a)(b\tens\I)\bigr)=(\h\tens\id)\bigl((a\tens\I)((\id\tens{S_\HH})\alpha(b))\bigr)
\]
for all $a,b\in\cA$.
\end{lemma}

\begin{proof}
It is a standard result that the maps
\[
\begin{split}
\cA\tens\Pol(\HH)\ni{a}\tens{c}&\longmapsto\alpha(a)(\I\tens{c})\in\cA\tens\Pol(\HH),\\
\cA\tens\Pol(\HH)\ni{a}\tens{c}&\longmapsto\bigl((\id\tens{S_\HH})\alpha(a)\bigr)(\I\tens{c})\in\cA\tens\Pol(\HH)
\end{split}
\]
(between algebraic tensor products) are mutually inverse isomorphisms of vector spaces. Therefore for any $b\in\cA$ we can find $x_1,\dotsc,x_n\in\cA$ and $y_1,\dotsc,y_n\in\Pol(\HH)$ such that
\begin{equation}\label{S1}
b\tens\I=\sum_i\alpha(x_i)(\I\tens{y_i})
\end{equation}
i.e.
\begin{equation}\label{S2}
\sum_ix_i\tens{y_i}=(\id\tens{S_\HH})\alpha(b).
\end{equation}

Thus, by \eqref{S1}, invariance of $\h$ and \eqref{S2}
\[
\begin{split}
(\h\tens\id)\bigl(\alpha(a)(b\tens\I)\bigr)&=(\h\tens\id)\sum_i\bigl(\alpha(a)\alpha(x_i)(\I\tens{y_i})\bigr)\\
&=(\h\tens\id)\sum_i\bigl(\alpha(ax_i)(\I\tens{y_i})\bigr)\\
&=\sum_i(\h\tens\id)\bigl(\alpha(ax_i)\bigr)y_i\\
&=\sum_i\h(ax_i)y_i\\
&=(\h\tens\id)\biggl((a\tens\I)\sum_ix_i\tens{y_i}\biggr)\\
&=(\h\tens\id)\bigl((a\tens\I)((\id\tens{S_\HH})\alpha(b))\bigr).
\end{split}
\]

\end{proof}

\section{Commutativity results}

\begin{theorem}
\label{slw}
Let $\sC$ be a unital \cst-algebra and let $\beta\colon\sA\to\sC\tens\sA$ and $\gamma\colon\hh{\sA}\to\hh{\sA}\tens\sC$ be unital $*$-homomorphisms such that
\[
(\id\tens\beta)W=(\gamma\tens\id)W.
\]
Then the algebra generated inside $\sC$ by
\[
\bigl\{(\id\tens\rh)\beta(a)\st\rh\in\sA^*,\:a\in\sA\bigr\}
\]
is commutative.
\end{theorem}

\begin{proof}
Let $V=(\id\tens\beta)W$. Note that the element $V\in\hh{\sA}\tens\sC\tens\sA$ satisfies
\[
(\hh{\Delta}\tens\id\tens\id)V=(\hh{\Delta}\tens\beta)W=(\id\tens\id\tens\beta)(W_{13}W_{23})=V_{134}V_{234}
\]
and
\[
(\id\tens\id\tens\Delta)V=(\gamma\tens\Delta)W=(\gamma\tens\id\tens\id)(W_{12}W_{13})=V_{123}V_{124}.
\]

Let us compute $(\hh{\Delta}\tens\id\tens\Delta)V$. On one hand
\[
(\hh{\Delta}\tens\id\tens\Delta)V=(\hh{\Delta}\tens\id\tens\id\tens\id)(V_{123}V_{124})=V_{134}V_{234}V_{135}V_{235}
\]
and on the other
\[
(\hh{\Delta}\tens\id\tens\Delta)V=(\id\tens\id\tens\id\tens\Delta)(V_{134}V_{234})=V_{134}V_{135}V_{234}V_{235}.
\]
Therefore
\begin{equation}\label{com0}
V_{234}V_{135}=V_{135}V_{234}.
\end{equation}

Take now $\ph_1,\ph_2\in\sA^*$ and $\mu,\nu\in\sC^*$ and let us apply $(\ph_1\tens\ph_2\tens\id\tens\mu\tens\nu)$ to both sides of \eqref{com0}. Denoting
\[
a=(\ph_1\tens\id)W\quad\text{and}\quad{b}=(\ph_2\tens\id)W
\]
we obtain from \eqref{com0}
\[
\bigl((\id\tens\nu)\beta(a)\bigr)\bigl((\id\tens\mu)\beta(b)\bigr)=\bigl((\id\tens\mu)\beta(b)\bigr)\bigl((\id\tens\nu)\beta(a)\bigr)
\]
which proves the claim.
\end{proof}

Let $\HH$ be the quantum automorphism group of $(\sA,\Delta)$ defined in \cite{bhss}. We have the distinguished (right) action $\alpha\colon\sA\to\sA\tens\C(\HH)$ of $\HH$ on $\sA$. It is a fact that $\HH$ is a compact quantum group of Kac type, i.e.~its antipode $S_\HH$ is an involutive $*$-antiautomorphism of $\C(\HH)$. This follows from e.g.~\cite[Theorem 2.10]{BS} because by results \cite[Section 2]{bhss} the action of $\HH$ on $\sA$ preserves the decomposition $\sA=\CC\I\oplus(\CC\I)^\perp$, where the orthogonal complement is defined by the scalar product given by $\h$. 

Another point to note is that $\C(\HH)$ is generated as a \cst-algebra by
\begin{equation}\label{genCH}
\bigl\{(\rh\tens\id)\alpha(a)\st\rh\in\sA^*,\:a\in\sA\bigr\}.
\end{equation}
This is not explicitly stated in \cite{bhss}, but this is clear in all cases of compact quantum group defined by similar universal properties to that of $\HH$. The \cst-algebra generated by \eqref{genCH} carries all the structure making it the universal object $\C(\HH)$.

In what follows we will make use of the \emph{Fourier transform} on the finite quantum group $(\cA,\Delta)$ defined as
\[
\cF\colon\cA\ni{a}\longmapsto\h(\cdot\,{a})\in\cA^*.
\]
The Fourier transform is an isomorphism of vector spaces with various additional properties. Note that $(\cG\comp\cF)(a)=\bigr(\id\tens\h(\cdot\,{a})\bigr)W$.

The results of \cite[Section 2]{bhss} show that the map $(\cF\tens\id)\comp\alpha\comp\cF^{-1}\colon\cA^*\to\cA^*\tens\C(\HH)$ is a unital $*$-homomorphism (in fact it is an \emph{action by automorphisms} of the quantum group dual to $(\cA,\Delta)$). This yields a unital $*$-homomorphism $\gamma$
\[
\gamma=(\cG\tens\id)\comp(\cF\tens\id)\comp\alpha\comp\cF^{-1}\comp\cG^{-1}\colon\hh{\sA}\longrightarrow\hh{\sA}\tens\C(\HH).
\]

Finally, since $S_\HH$ is a $*$-antiautomorphism (as is the antipode $S$ of $(\sA,\Delta)$) we can define a unital $*$-homomorphism $\gamma\colon\sA\to\C(\HH)\tens\sA$ by
\[
\beta=\sigma\comp(S\tens{S_\HH})\comp\alpha\comp{S},
\]
where $\sigma$ denotes this time the flip $\sA\tens\C(\HH)\to\C(\HH)\tens\sA$. By \cite[Lemma 2.2 and Proposition 2.3]{bhss} we have in fact
\begin{equation}\label{betaS}
\beta=\sigma\comp(\id\tens{S_\HH})\comp\alpha.
\end{equation}

Let us also note that
\[
\bigl\{(\id\tens\rh)\beta(a)\st\rh\in\sA^*,\:a\in\sA\bigr\}
\]
generates $\C(\HH)$ (it is exactly \eqref{genCH}).

\begin{theorem}\label{main}
We have
\[
(\id\tens\beta)W=(\gamma\tens\id)W.
\]
\end{theorem}

\begin{proof}
As $\cG\comp\cF$ appears quite often in our considerations let us denote this composition by $\varGamma$. Take $a\in\cA$ and let us compute
\[
\begin{split}
\bigl(\id\tens\id\tens\h(\cdot\,a)\bigr)
\bigl((\gamma\tens\id)W\bigr)&=\gamma\bigl((\id\tens\h(\cdot\,a))W\bigr)\\
&=\bigl(\varGamma\tens\id\bigr)\bigl(\alpha(a)\bigr)\\
&=(\id\tens\h\tens\id)\bigl(W_{12}\alpha(a)_{23}\bigr).
\end{split}
\]
Similarly, using \eqref{betaS} we find that
\[
\bigl(\id\tens\id\tens\h(\cdot\,a)\bigr)\bigl((\id\tens\beta)W\bigr)
=(\id\tens\h\tens\id)\bigl((\id\tens\id\tens{S_\HH}))(\id\tens\alpha)W)(\I\tens{a}\tens\I)\bigr).
\]

Thus the proof will be complete when we show that
\begin{equation}\label{complete}
(\id\tens\h\tens\id)\bigl(W_{12}\alpha(a)_{23}\bigr)=
(\id\tens\h\tens\id)\bigl((\id\tens\id\tens{S_\HH}))(\id\tens\alpha)W)(\I\tens{a}\tens\I)\bigr)
\end{equation}
for all $a\in\cA$.

Let us slice the first leg of both sides of \eqref{complete} with a functional $\h(\cdot\,b)$ ($b\in\cA$). We obtain
\begin{equation}\label{equiv}
(\h\tens\id)\bigl((b\tens\I)\alpha(a)\bigr)=(\h\tens\id)\bigl(((\id\tens{S_\HH})\alpha(b))(a\tens\I)\bigr).
\end{equation}
Note that \eqref{equiv} for all $a,b\in\cA$ is equivalent to \eqref{complete} for all $a\in\cA$. Therefore our theorem follows from Lemma \ref{strinv} and the fact that $\h$ is a trace (\cite[Proposition A.2.2]{pseudogr}).
\end{proof}

\begin{corollary}
Let $\GG=(\cA,\Delta)$ be a finite quantum group. Then the quantum group of automorphisms of $\GG$ defined in \cite{bhss} is the classical group of Hopf $*$-algebra automorphisms of $\cA$. In particular any compact quantum group acting faithfully on $\GG$ is in fact a classical group.
\end{corollary}

\begin{proof}
As we noted just before stating Theorem \ref{main}, the slices of $\beta$ generate the \cst-algebra $\C(\HH)$, where $\HH$ is the quantum group of automorphisms of $\GG$. By Theorem \ref{slw} this algebra is commutative. The last statement of the theorem follows from the universal property of the quantum group $\HH$ and \cite[Proposition 2.7]{bhss}.
\end{proof}

\section*{Acknowledgements}

The authors would like to thank Biswarup Das, Uwe Franz and Adam Skalski for many stimulating discussions on the topic of quantum groups of automorphisms of finite quantum groups.

\end{document}